\documentclass[11pt, oneside]{article}
\usepackage[utf8]{inputenc}
\usepackage{amsmath, amsthm, amssymb,xcolor}
\usepackage{enumerate, comment, url}
\usepackage{dsfont}%for making characteristic function

\usepackage{todonotes}
\usepackage{multicol}
\usepackage{tikz}
\usetikzlibrary{decorations.markings}
\usepackage{standalone}
\usepackage{mathtools}
\usepackage{adjustbox}
\usepackage[left=1in,right=1in,top=1in,bottom=1in,bindingoffset=0cm]{geometry}

\usepackage[normalem]{ulem}
\usepackage{enumitem}

\usepackage{hyperref}
\hypersetup{colorlinks=true,linkcolor=black,anchorcolor=blue,citecolor=black,filecolor=blue,urlcolor=blue,bookmarksnumbered=true,pdfview=FitB}

\usepackage[capitalise]{cleveref}

\newtheorem{theorem}{Theorem}
\newtheorem{lemma}[theorem]{Lemma}

\newtheorem{conjecture}[theorem]{Conjecture}

\newtheorem*{theorem*}{Theorem}
\newtheorem*{corollary*}{Corollary}

\theoremstyle{definition}

\DeclareMathOperator{\mad}{mad}

\title{On orientations with forbidden out-degrees}

\author{
Owen Henderschedt\footnote{Auburn University, Department of Mathematics and Statistics, Auburn U.S.A.
  Email: {\tt olh0011@auburn.edu}.}
\and
Jessica McDonald\footnote{Auburn University, Department of Mathematics and Statistics, Auburn U.S.A.
  Email: {\tt mcdonald@auburn.edu}.   
	Supported in part by Simons Foundation Grant \#845698  }}
\date{}

\begin{document}
\maketitle

\begin{abstract} Let $G$ be a $d$-regular graph and let $F\subseteq\{0, 1, 2, \ldots, d\}$ be a list of forbidden out-degrees. Akbari, Dalirrooyfard, Ehsani, Ozeki, and Sherkati conjectured that if $|F|<\tfrac{1}{2}d$,  then $G$ should admit an $F$-avoiding orientation, i.e., an orientation where no out-degrees are in the forbidden list $F$. The conjecture is known for $d\leq 4$ due to work of Ma and Lu, and here we extend this to $d\leq 6$. The conjecture has also been studied in a generalized version, where $d, F$ are changed from constant values to functions $d(v), F(v)$ that vary over all $v\in V(G)$. We provide support for this generalized version by verifying it for some new cases, including when $G$ is 2-degenerate and when every $F(v)$ has some specific structure.
\end{abstract}

\section{Introduction}

In this paper all graphs are loopless, although parallel edges are permitted. We follow \cite{WestText} for standard terms not defined here.

An \emph{orientation} $D$ of a graph $G$ is an assignment of direction to each edge in $G$, so that an edge points \emph{out} of one of its endpoints and \emph{in} to its other endpoint. For a given orientation, the number of edges pointing out of a vertex is its \emph{out-degree} (and the number pointing in is its \emph{in-degree}). In this paper we are concerned about the existence of orientations where certain out-degree values are forbidden. 
Consider the following conjecture.

\begin{conjecture} [Akbari, Dalirrooyfard, Ehsani, Ozeki, and Sherkati \cite{Avoid}]\label{conj: regConst} Let $G$ be a $d$-reglar graph and let $F\subseteq \{0, 1, \ldots, d\}$ be a list of forbidden out-degrees. If $|F|< \tfrac{1}{2}d$, then $G$ admits an $F$-avoiding orientation, that is, an orientation where no out-degrees are in the forbidden list $F$.
\end{conjecture}

Conjecture \ref{conj: regConst} is tight if true -- in particular, the authors of the conjecture proved in \cite{Avoid} that $K_{2k+1}$ has no $F$-avoiding orientation for $F=\{k, \ldots, 2k-1\}$. They also proved that their conjecture holds for all bipartite graphs and all cliques. In terms of specific $d$, Conjecture \ref{conj: regConst} is trivial for $d=1, 2$ (where the forbidden list is empty), and true for $d=3,4$ (see Ma and Lu \cite{no-consecutive}). Here we verify the next cases of $d=5,6$ with the following theorem.

\begin{theorem}\label{thm: 5+-reg}
Let $G$ be a $d$-regular graph with $d\geq 5$, and let $F\subseteq\{0, 1, 2, \ldots, d\}$ with $|F|\leq 2$. Then $G$ admits an $F$-avoiding orientation.
\end{theorem}

Our proof of Theorem \ref{thm: 5+-reg} appears in Section 4 of this paper, and we shall say more about it shortly.

Conjecture \ref{conj: regConst} can be generalized in two directions, by changing $d, F$ from constant values to functions $d(v), F(v)$ that vary over every $v\in V(G)$. Given an orientation $D$ of a graph $G$, the degree $d(v)$ of a vertex $v\in V(G)$ is partitioned into out-degree $d_D^+(v)$ and in-degree $d_D^-(v)$ in $D$. Given a function $F: V(G) \to 2^{\mathbb{N}}$ of forbidden values for each vertex, we say that the orientation $D$ is \emph{$F$-avoiding} if $d^+_D(v)\notin F(v)$ for all $v\in V(G)$.

\begin{conjecture}[Akbari, Dalirrooyfard, Ehsani, Ozeki, and Sherkati \cite{Avoid}]\label{conj: F-avoiding}
    Let $G$ be a graph and let $F:V(G) \to 2^{\mathbb{N}}$. If $|F(v)| < \tfrac{1}{2}d(v)$ for all $v\in V(G)$, then $G$ admits an $F$-avoiding orientation.
\end{conjecture}

The known tight examples for Conjecture \ref{conj: F-avoiding} are the same as those for Conjecture \ref{conj: regConst}. On the other hand, while bipartite graphs were shown to also satisfy the more general conjecture in \cite{Avoid} (in fact even with list sizes equal to $\tfrac{1}{2}d(v)$), cliques were not, and indeed Conjecture \ref{conj: F-avoiding} is still unproven for cliques of size $6$ and greater. Here, we add to the families known to satisfy Conjecture \ref{conj: F-avoiding} with the following two results.

\begin{theorem}\label{thm: 2-deg} Let $G$ be a 2-degenerate graph and let $F:V(G)\to 2^{\mathbb{N}}$ with $F(v)<\tfrac{1}{2}d(v)$ for all $v\in V(G)$. Then $G$ has an $F$-avoiding orientation.
\end{theorem}

\begin{theorem}\label{thm: bipartite-plus-extra}
Let $G$ be a graph and let $F:V(G)\to 2^{\mathbb{N}}$ where $|F(v)|<\frac{1}{2}d(v)$ for all $v\in V(G)$. Suppose that $G$ can be decomposed into a bipartite graph and a subgraph $H$ with $\Delta(H)\leq 2$ and such that every vertex $v\in V(H)$ with $d_H(v)=2$ has $d_G(v) \equiv 0 \pmod{2}$.  Then $G$ admits an $F$-avoiding orientation.
\end{theorem}

Note that by a 2-degenerate graph we mean a graph where every subgraph contains a vertex of degree at most $2$. Hence the class of 2-degenerate graphs includes all 1-degenerate graphs (i.e. all forests). It also includes all planar graphs of girth at least 6. (To see this, note that for a planar graph $G$ of girth $g$, Euler's formula gives $(\mad(G)-2 ) ( g-2 ) < 4$, where $\mad(G)$ is  the \emph{maximum average degree} of $G$ defined by $\mad(G)=\max\{\tfrac{2|E(H)|}{|V(H)|}: H\subseteq G \}$). Our proof of Theorem \ref{thm: 2-deg}, which is in Section 2, involves arguing about smallest possible counterexamples; as part of Section 2 we also prove a general list of properties that any smallest counterexample to Conjecture \ref{conj: F-avoiding} must have. The same sort of arguments allow us to bootstrap the afore-mentioned result for bipartite graphs from \cite{Avoid} into Theorem \ref{thm: bipartite-plus-extra}, the proof of which is also contained in Section 2.

There has been more success with approximations to Conjecture \ref{conj: F-avoiding} than verifications for special families.  In their initial paper, Akbari et al.\cite{Avoid} showed that replacing $\tfrac{1}{2}$ with $\tfrac{1}{4}$ results in a true statement. Recently Bradshaw, Chen, Ma, Mohar, and Wu \cite{Mohar} improved this upper bound to $\lfloor\frac{1}{3}d_G(v)\rfloor$ using the combinatorial nullstellensatz of Alon and Tarsi \cite{NULL-TARSI}\cite{NULL}. 

Another approach to Conjecture \ref{conj: F-avoiding} has been to consider special sorts of lists. To this end, let $G$ be a graph with forbidden lists $F:V(G)\to 2^{\mathbb{N}}$. For every vertex $v\in V(G)$ we can partition $\{0,\ldots, d(v)\}$ into maximal intervals in $F(v)$, which we call \emph{holes}, and maximal intervals not in $F(v)$, which we call \emph{homes}.  Ma and Lu \cite{no-consecutive} proved that given a graph $G$ and  function $F:V(G)\to 2^{\mathbb{N}}$, if every hole has size at most one then $G$ admits an $F$-avoiding orientation. Note that lists of this sort  may satisfy $|F(v)|\sim \tfrac{1}{2}d_G(v)$ for all $v$. In this paper we take another step in this direction and allows holes of size two, given that the overall ratio is somewhat worse, and given a condition on the \emph{end-intervals}. The \emph{end-intervals} for a vertex $v$ are those holes or homes containing $0$ or $d(v)$; note that each vertex has two end-intervals but they need not be distinct.

\begin{theorem}\label{thm: 2hole3home}
Let $G$ be a graph and let $F: V(G)\to 2^{\mathbb{N}}$ be such that for each $v\in V(G)$: all holes have size at most two; between any pair of distinct holes is a home of size at least 3, and both end-intervals are homes of size at least two. Then $G$ admits an $F$-avoiding orientation. 
\end{theorem}

Observe that in \cref{thm: 2hole3home}, the lists may satisfy
$|F(v)| \sim \tfrac{2}{5}d_G(v)$, beating the $\tfrac{1}{3}$ bound. Our proof of \cref{thm: 2hole3home} uses a new tool for modifying orientations called a \emph{lasso}; we will introduce this concept and prove \cref{thm: 2hole3home} in Section 4. In Section 3 we prove the following two results about very special forbidden lists that satisfy Conjecture \ref{conj: F-avoiding}.

\begin{theorem}\label{thm: FGforb}
    Let $G$ be a graph and let $F: V(G)\to 2^{\mathbb{N}}$ where $|F(v)|\leq \frac{1}{2}d(v)$ for all $v\in V(G)$. If every hole is an end-interval, then $G$ admits an $F$-avoiding orientation.
\end{theorem}

\begin{theorem}\label{lem: extreme}
    Let $k\in\mathbb{Z}^+$. Every $(2k+1)$-regular graph admits a $\{1,2,...,k\}$-avoiding orientation.
\end{theorem}

We will see that that proof of Theorem \ref{thm: FGforb} follows from a classic theorem on orientations due to Frank and Gyarfas \cite{FG} (also Hakimi \cite{Hak}), and Theorem \ref{lem: extreme} follows from a result about maximum directed cuts due to Along, Bollob\'{a}s, Gy\'{a}rf\'{a}s, Lehel and Scott \cite{AB}. Theorem \ref{lem: extreme} is interesting in that it is close to the tightness result of Akbari et al. \cite{Avoid} mentioned above. The reason we prove it here however it that it is a needed special case towards our proof of Theorem \ref{thm: 5+-reg}. We will see that the other needed cases come from Theorem \ref{thm: 2hole3home} and our new lasso technique.

\section{Minimum counterexamples and special graph classes}\label{sec: minimum counterexample}

We prove the following properties about any minimum counterexample to Conjecture \ref{conj: F-avoiding}.

\begin{lemma}\label{minCounter}
Let $G$ be a graph and let $F:V(G)\to 2^{\mathbb{N}}$ with $F(v)<\tfrac{1}{2}d_G(v)$ for all $v\in V(G)$. Suppose that $G$ does not admit an $F$-avoiding orientation, and suppose that $|E(G)|$ is minimum subject to this. Then:
 \begin{enumerate}[label = {(\arabic*)}]
     \item[(a)] all vertices of even degree in $G$ form an independent set;
     \item[(b)] if $v$ is a vertex of even degree in $G$, then $u\in N_G(v)$ implies $0, d_G(u)\not\in F(u)$; and
    \item[(c)] if $u,v$ are any pair of adjacent vertices in $G$ then either $0\not\in F(u)$ or $d_G(v)\not\in F(v)$.
\end{enumerate}
 Moreover, if $|E(G)|+|V(G)|$ is minimum, then $\delta(G) \geq 3$.
\end{lemma}

\begin{proof} Define the following sets of vertices:
$$A_1=B_1=\{v\in V(G) : v \text{ has even degree}\};$$
$$A_2=\{v\in V(G) : d_G(v)\in F(v)\};\hspace*{.2in} B_2=\{v\in V(G) : 0\in F(v)\};$$
$$\mathcal{A}=A_1\cup A_2; \hspace*{.2in}\textrm{and } \hspace*{.2in}\mathcal{B}=B_1\cup B_2.$$
Proving (a) means proving that there is no edge between $A_1$ and $B_1$, proving (b) means proving there is no edge between $A_1$ and $B_2$ and no edge between $B_1$ and $A_2$, and proving (c) means proving there is no edge between $A_2$ and $B_2$. We will handle all these case together by supposing, for a contradiction, that there exists an edge $e$ joining $u\in \mathcal{A}$ with $v\in\mathcal{B}$.

Let $G' = G-e$. We define $F':V(G')\to 2^{\mathbb{N}}$ as follows:
   \begin{equation*}
        F'(w) = \begin{cases}
            F(w) \hspace{.125cm} &\text{if $w\notin \{u,v\}$} \\
            F(w) \hspace{.125cm} &\text{if $w=u$ and $u\in A_1$}\\
            F(w)\setminus \{d_G(w)\} \hspace{.125cm} &\text{if $w=u$ and $u\in A_2$}\\
            \{i-1 : i\in F(v), i\geq 1\} \hspace{.125cm} &\text{if  $w=v$.}
        \end{cases}
    \end{equation*}
We will show that
\begin{equation}\label{Fcond}
|F'(w)|< \tfrac{1}{2}d_{G'}(w) \hspace*{.2in}\textrm{for all $w\in V(G')$}.
\end{equation}
To this end, note that for a vertex $w\not\in \{u, v\}$, $d_G'(w)=d_G(w)$ and $F'(w)=F(w)$ so condition (1) is trivially satisfied for such $w$. If $w=u$ and $u\in A_2$, then  $d_G'(w)= d_G(w)-1$ and $|F'(w)|=|F(w)|-1$. Satisfying (1) therefore amounts to $|F(w)|-1 <\tfrac{1}{2}\left(d_G(w)-1\right)$, or equivalently, $|F(w)| < \tfrac{1}{2}d_G(w)+\tfrac{1}{2}$, which we 
know by assumption.  If $w=u$ and $u\in A_1$, then $F'(w)=F(w)$ while $d_G'(w)< d_G(w)$, but since $d_G(w)$ is even, (1) follows from $|F(w)|< \tfrac{1}{2}d_{G}(w)$. Finally, consider that $w=v$, where we know that $d_G'(w)< d_G(w)$. If $0\in F(v)$ then $|F'(v)| = |F(v)|-1$, and the same computation we did in the $A_2$-case verifies (1). Otherwise, since $v\in \mathcal{B}$, we know that $v\in B_1$ and hence that $v$ has even degree is $G$, so we get (1) analogously to the $A_1$-case.

Since our $F':V(G')\to 2^{\mathbb{N}}$ satisfies (1) and since $G$ is edge-minimal, 
we get an $F'$-avoiding orientation $D'$ of $G'$. Define an orientation $D$ of $G$ from $D'$ by orienting $e$ from $v$ to $u$.  Then
$d^+_D(w)=d^+_{D'}(w)\notin F'(w)=F(w)$ for any $w\notin \{u,v\}$. Note that $d_D^+(u)=d_{D'}(u)$ since $v$ points towards $u$. If $u\in A_1$ then $F'(u) = F(u)$ so $d_D^+(u)\not\in F(u)$; if $u\in A_2$ then $F'(u)=F(u)\setminus \{d_G(u)\}$ and $d^+_D(u) \neq d_G(u)$ , so we have $d^+_D(u)\notin F(u)$. Finally consider $v$: since $d_D^+(v) -1 = d^+_{D'}(v) \in \{i-1: i\in F(v), i\geq 1\}$, we have $d^+_D(v)\notin F(v)$. Thus, $D$ is an $F$-avoiding orientation of $G$, contradiction. This completes our proof of (a), (b), and (c).

In order to prove (d), suppose for a contradiction that there exists $v_0\in V(G)$ with $d_G(v_0)\leq 2$. Since $|F(v_0)| < \frac{1}{2}(2)$, $F(v_0) = \emptyset$. We cannot have $d_G(v_0)=0$, as isolates cannot satisfy this strict inequality, so $d_G(v_0)\in\{1, 2\}$. We divide our proof into two cases. \\

\noindent\textbf{Case 1:} \emph{Every $w\in N_G(v_0)$ has $d_G(w)\in F(w)$.}

Note that in this case, we know that $d_G(v_0)=1$, by part (b) above; let $u_0$ be the lone neighbour of $v_0$ in $G$. 
Define $G' = G - v_0$ and define forbidden lists $F'$ for $G'$ by : $F'(w)=F(w)$ if $w\neq u_0$, and $F'(u_0)= F(u_0) \setminus \{d_G(u_0)\}$.
Clearly $|F'(w)|=|F(w)| < \frac{1}{2}d_G(w) = \frac{1}{2}d_{G'}(w)$ for all $w\neq u_0$. For $u_0$ we have
$$|F'(u_0)|=|F(u_0)|-1 < \tfrac{1}{2}d_G(u_0)-1 = \tfrac{1}{2}(d_{G'}(u_0) +1)-1<\tfrac{1}{2}d_{G'}(u_0).$$
So by minimality there exists an orientation $D'$ of $G'$ which is $F'$-avoiding. Define an orientation $D$ from $D'$ by directing that last edge from $v_0$ to $u_0$. This means that $d^+_D(w)=d^+_{D'}(w)$ for all $w\in V(G)\setminus \{v_0\}$. Recall that $F(v_0)=\emptyset$, so $D$ trivially satisfies the $F$-avoiding condition for $v_0$. Since $F'(w)=F(w)$ for all $w\neq u_0$, the $F$-avoiding condition is satisfied for all such $w$. 
For $u_0$ there is one extra forbidden value in $F(u_0)$ as compared to $F'(u_0)$, but this extra value is $d_G(u_0)$, which is certainly not the out-degree of $u_0$ in $D$ since the edge between $v_0$ and $u_0$ points into $u_0$. Hence $D$ is an $F$-avoiding orientation of $G$, contradiction.  \\

\noindent\textbf{Case 2:} \emph{There exists some $u\in N_G(v_0)$ with $d_G(u)\not\in F(u)$.}

In particular, the assumption of this case implies that there exists some $\alpha\in F(u)$ such that $\alpha+1\notin F(u)$.
Let $\tilde{G}$ be the graph obtained from $G$ by deleting one edge, between $u$ and $v_0$, and also deleting $v_0$ if the edge-deletion makes it isolated. We define the following forbidden lists $\tilde{F}$ of $\tilde{G}$.
\begin{equation*}
    \tilde{F}(w) = \begin{cases}
        F(w) \hspace{.5cm} &\text{if $w\in V(G)$,  $w\neq u$}\\
        F(w)\setminus\{\alpha\} \hspace{.5cm} &\text{if $w = u$}
    \end{cases}
\end{equation*}
For any $w\notin \{u,v_0\}$ we have $|\tilde{F}(w)|=|F(w)|< \tfrac{1}{2}d_{G}(w) = \tfrac{1}{2}d_{G'}(w).$ For $u$ we get
$$|\tilde{F}(u)|=|F(u)|-1 <\tfrac{1}{2}d_G(u)-1 = \tfrac{1}{2}(d_{\tilde{G}}(u)+1)-1 < \tfrac{1}{2}d_{\tilde{G}}(u).$$ 
If $v_0\in V(\tilde{G})$, then $v_0$ is not isolated so $|F(v_0)|=|\tilde{F}(v_0)|=0$ implies that $|\tilde{F}(v_0)|< \tfrac{1}{2}d_{\tilde{G}}(v_0)$.

By minimality there exists an exists an $\tilde{F}$-avoiding orientation $\tilde{D}$ of $\tilde{G}$. We now obtain an orientation $D$ from $\tilde{D}$ by orienting our deleted edge as follows. If $d^+_{\tilde{D}}(u) = \alpha$, orient the edge from $u$ to $v_0$. Otherwise orient from $v_0$ to $u$. Clearly $d^+_{D}(w)=d^+_{\tilde{D}}(w) \notin \tilde{F}(w) = F(w)$ for all $w\neq \{u,v_0\}$. Since $F(v_0)=\emptyset$, in order to show that $D$ is an $F$-avoiding orientation of $G$ (and get our desired contradiction) it remains only to check that $D$ is $F$-avoiding at $u$. If $d^+_{\tilde{D}}(u) = \alpha$, then $d^+_D(u) = \alpha +1 \notin F(u)$. On the other hand, if $d^+_{\tilde{D}}(u)\neq \alpha$ then $d^+_{D}(u)= d^+_{\tilde{D}}(u)$. Since $\tilde{F}(u)\subset F(u)$, we also get that $d^+_D(u) \notin F(u)$ in this situation. Hence $D$ is an $F$-avoiding orientation of $G$, contradiction.
\end{proof}

Our work in Lemma \ref{minCounter}, and in particular the last sentence of the lemma statement, immediately implies that Conjecture \ref{conj: F-avoiding} holds for 2-degenerate graphs. 

\setcounter{theorem}{3}

\begin{theorem} Let $G$ be a 2-degenerate graph and let $F:V(G)\to 2^{\mathbb{N}}$ with $F(v)<\tfrac{1}{2}d_G(v)$ for all $v\in V(G)$. Then $G$ has an $F$-avoiding orientation.
\end{theorem}

\begin{proof} Suppose not, and let $G$ be a counterexample where $|V(G)|+|E(G)|$ is smallest. Since $G$ has a vertex of degree at most 2, and so does every subgraph of $G$, we get our desired result by repeating the $\delta \geq 3$ part of the proof of Lemma \ref{minCounter}.
\end{proof}

It is tempting to want to say that if $G$ is a graph where all even-degree vertices form an independent set (or $G$ doesn't have any even-degree vertices), then $G$ satisfies Conjecture \ref{conj: F-avoiding}, but this does not follow from Lemma \ref{minCounter} because such a class is not closed under taking subgraphs. On the other hand, we can use arguments similar to those above to expand further the class of graphs which are known to satisfy Conjecture \ref{conj: F-avoiding}.

\begin{theorem}\label{thm: bipartite-plus-extra}
Let $G$ be a graph and let $F:V(G)\to 2^{\mathbb{N}}$ where $|F(v)|<\frac{1}{2}d_G(v)$ for all $v\in V(G)$. Suppose that $G$ can be decomposed into a bipartite graph and a subgraph $H$ with $\Delta(H)\leq 2$ and such that every vertex $v\in V(H)$ with $d_H(v)=2$ has $d_G(v) \equiv 0 \pmod{2}$.  Then $G$ admits an $F$-avoiding orientation.
\end{theorem}
\setcounter{theorem}{9}

\begin{proof} We may assume without loss of generality that $H$ is spanning, since adding vertices of degree zero to $H$ has no effect on our assumptions.
If all vertices have degree 0 or 2 in $H$, then $H$ is a collection of cycles and isolates, and by consistently orienting the cycles we get an orientation $D_H$ of $H$ with $d^+_{D_H}(v)= \{0,1\}$ for all $v\in V(H)$. If $H$ is not Eulerian, then by adding one dummy vertex and joining it to every odd-degree vertex of $H$, we get a collection of cycles (and isolates) $H'$ -- by consistently orienting the cycles in $H'$ and then deleting the dummy vertex, we again get an 
orientation $D_{H}$ of $H$ with $d^+_{D_{H}}(v)\in\{0,1\}$ for all $v\in V(H)$. It remains now to orient every edge in $E(G)\setminus E(H)$, and show that our overall orientation is $F$-avoiding.

Let $G' = G-E(H)$. We define forbidden lists $F'$ of $G'$ as follows.
           \begin{equation*}
        F'(w) = \begin{cases}
            F(w) \hspace{.125cm} &\text{ if $d^+_{D_H}(w)=0$} \\
            \{i-1 : i\in F(w), i\geq 1\} \hspace{.125cm} &\text{ if $d^+_{D_H}(w) = 1$}
        \end{cases}
    \end{equation*}
For any $w\in V(G)$ with $d^+_{D_H}(w) = 0$, we have $|F'(w)|=|F(w)| < \tfrac{1}{2}d_G(w) = \tfrac{1}{2}d_{G'}(w).$ Now suppose $d^+_{D_H}(w) = 1$ and hence $d_H(w)\in \{1,2\}$. If $d_H(w) = 1$ then 
\[|F'(w)|\leq |F(w)|<\tfrac{1}{2}d_G(w) = \tfrac{1}{2}(d_{G'}(w)+1)\] 
which implies $|F'(w)| \leq \frac{1}{2}d_{G'}(w)$. Finally if $d_H(w) = 2$, then 
\[|F'(w)|\leq |F(w)|<\tfrac{1}{2}d_G(w) = \tfrac{1}{2}(d_{G'}(w)+2).\] 
This implies $|F'(w)|\leq \frac{1}{2}(d_G'(w)+1)$ but since $d_{G'}(w)\equiv d_{G}(w)\equiv 0 \pmod{2}$ in this case, $|F'(w)|\leq \frac{1}{2}d_{G'}(w)$. In all cases, $|F'(v)|\leq \frac{1}{2}d_{G'}(v)$ for all $v\in V(G')$, and since $G'$ is bipartite, there exists an $F'$-avoiding orientation $D'$ of $G'$ by the above-mentioned result of Akbari et al. \cite{Avoid}. We claim $D:= D' \cup D_H$ is an $F$-avoiding orientation of $G$. 

Clearly for any $v\in V(G)$ with $d^+_{D_H} = 0$ then $d^+_{D'}(v)\notin F'(v)=F(v)$. Now consider $v\in V(G)$ with $d^+_{D_H} = 1$. We have $d^+_D(v)-1 = d^+_{D'} \notin F'(v) = \{i: i-1\in F(v), i\geq 1\}$, so we have $d^+_D(v)\notin F(v)$. Thus, $D$ is an $F$-avoiding orientation of $G$.
\end{proof}

While it is still uncertain whether \cref{conj: F-avoiding} holds true for $K_6$, \cref{thm: bipartite-plus-extra} validates it for $K_6$ with any matching of size at least $2$ removed (since such a graph can be decomposed into copies of $K_{2,4}$, $K_2$ and $C_4$).

\section{Two results about specialized lists}\label{sec:special list}

The following theorem is a classic result of Frank and Gy\'arf\'as \cite{FG} on orientations of graphs (which also generalizes prior work of Hakimi \cite{Hak}) .

\begin{theorem}[Frank and Gy\'arf\'as \cite{FG}]\label{thm_FG} Let $G$ be a graph and let $\ell,u: V(G) \rightarrow \mathbb{N}$ with $\ell\leq u$. Then $G$ has an orientation $D$ with $\ell(v)\leq d^+_D(v) \leq u(v)$ for all $v\in V(G)$ if and only if $\forall$ $S\subseteq V(G)$,
$$
\ell(S) \leq e[S]+\delta(S) \hspace*{.2in}\text{and}\hspace*{.2in} e[S]\leq u(S),$$
where $\ell(S) = \sum_{x\in S}\ell(x)$, $u(S) = \sum_{x\in S}u(x)$, $e[S]=|E(G[S])|$, and $\delta(S) = |\{uv\in E(G) : u\in S, v\notin S\}|$
\end{theorem}

Theorem \ref{thm_FG} immediatley gives us the following result towards Conjecture \ref{conj: F-avoiding}, where in fact the strict inequality of the conjecture is replaced by $\leq$.

\setcounter{theorem}{6}

\begin{theorem}
    Let $G$ be a graph and let $F: V(G)\to 2^{\mathbb{N}}$ where $|F(v)|\leq \frac{1}{2}d(v)$ for all $v\in V(G)$. If every hole is an end-interval, then $G$ admits an $F$-avoiding orientation.
\end{theorem}
\setcounter{theorem}{10}

\begin{proof} Since every hole is an end-interval, each vertex has a single home interval; for every $v\in V(G)$, let $\ell(v)$ be the smallest value in this home interval and let $u(v)$ be the largest. Then for every $v\in V(G)$, $F(v)=\{0, \ldots, \ell(v)-1\}\cup \{u(v)+1, \ldots, d(v)\}$, with $|F(v)|=\ell(v)+d(v)-u(v)$. Since $|F(v)|\leq\tfrac{1}{2}d(v)$ and $0\leq \ell(v)\leq u(v) \leq d(v)$, this implies that $\ell(v) \leq \tfrac{1}{2}d(v)$ and $u(v)\geq\tfrac{1}{2}d(v)$ for all $v\in V(G)$.

Suppose for contradiction that $G$ does not admit an $F$-orientation. Then by \cref{thm_FG}, there exists some $S\subseteq V(G)$ with either 
    \begin{equation}\label{FG-interval option1}
    \ell(S) > e[S]+\delta(S)  
    \end{equation}
    or
    \begin{equation}\label{FG-interval option2}
    u(S)<e[S] .   
    \end{equation}
First suppose that $S$ realizes (\ref{FG-interval option1}). Since $\ell(v)\leq\frac{1}{2}d(v)$ for all $v\in V(G)$, we get 
\[\tfrac{1}{2}\sum_{x\in S}d_G(x) \geq \ell(S) > e[S] +\delta(S)= \tfrac{1}{2}\left(\sum_{x\in S}d_G(x) - \delta(S)\right) + \delta(S)=\tfrac{1}{2}\left(\sum_{x\in S}d_G(x) + \delta(S)\right) ,\] 
which is a contradiction. We can argue similarly in the case where $S$ realizes (\ref{FG-interval option2}). Here, since $u(v)\geq\tfrac{1}{2}d(v)$ for all $v\in V(G)$, we get  
\[\tfrac{1}{2}\sum_{x\in S}d_G(v) \leq u(S) < e[S] = \tfrac{1}{2}\left(\sum_{x\in S}d_G(x) - \delta(S)\right),\] 
which is also a contradiction.
\end{proof}

Consider the following theorem about orientations where for every vertex, at least one of the oudegree or in-degree is bounded. 

\begin{lemma}[Along, Bollob\'{a}s, Gy\'{a}rf\'{a}s, Lehel and Scott \cite{AB}]\label{lem: extreme-orientations} Let $a, b$ be non-negative integers. $G$ admits an orientation $D$ with either $d^+_D(v)\leq a$ or $d^-_D(v)\leq b$ for all $v\in V(G)$ if an only if there exists a partition $(V_1,V_2)$ of $V(G)$ with $\mad(G[V_1])\leq 2a$ and $\mad(G[V_2])\leq 2b$.
\end{lemma}

We get the following as a corollary.

\setcounter{theorem}{7}
\begin{theorem}
    Let $k\in\mathbb{Z}^+$. Every $(2k+1)$-regular graph admits a $\{1,2,...,k\}$-avoiding orientation.
\end{theorem}
\setcounter{theorem}{11}

\begin{proof} Suppose for contradiction that $G$ is a $(2k+1)$-regular graph that does not admit a $\{1,2,...,k\}$-avoiding orientation.  This is equivalent to saying that $G$ does not have an orientation $D$ satisfying, for all $v\in V(G)$: either $d^+_D(v)\leq 0$ or $d^+_D(v) \geq k+1$. Note that since $G$ is $(2k+1)$-regular this last condition can be replaced by $d^-_D(v)\leq k$. Then by Theorem \ref{lem: extreme-orientations}, for any partition $(V_1,V_2)$ of $V(G)$, either $\mad(G[V_1])> 0$ or $\mad(G[V_2])> 2k$. 

Let $V_1$ be a largest independent set in $G$ and let $V_2 = V(G)\setminus V_1$. Clearly $\mad(G[V_1]) = 0$ so we must have $\mad(G[V_2]) > 2k$. This implies that there exists $v\in V_2$ such that $d_{G[V_2]}(v) = 2k+1$. But since $G$ is $(2k+1)$-regular, $V_1 \cup \{v\}$ is a larger independent set contradicting the choice for $V_1$. Thus, $G$ admits a $\{1,2,...,k\}$-avoiding orientation.
\end{proof}

\section{The lasso and proofs of Theorems \ref{thm: 5+-reg} and \ref{thm: 2hole3home}}

Given a graph $G$ and forbidden list $F:V(G)\to 2^{\mathbb{N}}$, for any orientation $D$ of $G$ we define $D_i = \{v\in V(G): d^+_D(v)=i\}$ and $D_F = \{v\in V(G) : d^+_D(v)\in F(v)\}$. Thus, an orientation $D$ is $F$-avoiding if and only if $D_F=\emptyset$.

A common proof strategy when working with orientations is to find a directed path and then reverse the direction of all of its edges; note that in this modification only the endpoints of the path change their out-degrees. In this section we find another directed subgraph, which we'll call a \emph{lasso}, which can also be used to alter out-degrees in a controlled way. 

Consider the graph obtained from the path $(v_1, \ldots, v_k)$ by adding a single edge between $v_k$ and $v_i$ for some $2\leq i\leq k-1$. The directed graph obtained by orienting the path consistently from $v_1$ to $v_k$ and from $v_k$ to $v_i$ is called an \emph{out-lasso}; the directed graph obtained from an out-lasso by reversing the direction of every edge is called an \emph{in-lasso}. In either case, we call the directed graph a \emph{lasso} and denote it by $L=(v_1, \ldots v_k; v_i)$; we say that $L$ \emph{starts} at $v_1$. We \emph{flip} $L$ by reversing the orientation of the edges in $(v_1, \ldots, v_i)$ and $v_kv_i$; we call the result a \emph{flipped lasso} $L'=(v_i; v_1, \ldots, v_k)$; if $L$ is an out-lasso (in-lasso) we may also call $L'$ a \emph{flipped out-lasso (flipped in-lasso)}.  See figure \cref{fig:lassodef} for an example.

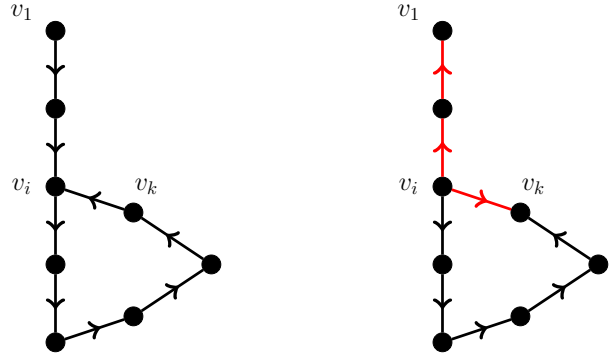
\begin{figure}[htb]
\centering
\begin{tikzpicture}

  \node[circle, draw, fill=black, text=white] (v1) at (0, 0) {};
  \node[circle, draw, fill=black, text=white] (v2) at (0, -1.5) {};
  \node[circle, draw, fill=black, text=white] (v3) at (0, -3) {};
  \node[circle, draw, fill=black, text=white] (v4) at (0, -4.5) {};
  \node[circle, draw, fill=black, text=white] (v5) at (0, -6) {};

  \node[circle, draw, fill=black, text=white] (v6) at (1.5, -5.5) {};
  \node[circle, draw, fill=black, text=white] (v7) at (3, -4.5) {};
  \node[circle, draw, fill=black, text=white] (v8) at (1.5, -3.5) {};

  \begin{scope}[very thick,decoration={
    markings,
    mark=at position 0.6 with {\arrow[scale = 1.5]{>}}}
    ] 
  
  \foreach \i/\j in {1/2,2/3,3/4,4/5,5/6,6/7,7/8,8/3}
    \draw[line width=1.5pt, postaction = {decorate}] (v\i) -- (v\j);

    \end{scope}

\node at (-.65,.35) {\Large{$v_1$}};
\node at (-.65,-3) {\Large{$v_i$}};
\node at (1.75,-3) {\Large{$v_k$}};
\end{tikzpicture}
\hspace{2cm}
\begin{tikzpicture}
  \node[circle, draw, fill=black, text=white] (v1) at (0, 0) {};
  \node[circle, draw, fill=black, text=white] (v2) at (0, -1.5) {};
  \node[circle, draw, fill=black, text=white] (v3) at (0, -3) {};
  \node[circle, draw, fill=black, text=white] (v4) at (0, -4.5) {};
  \node[circle, draw, fill=black, text=white] (v5) at (0, -6) {};

  \node[circle, draw, fill=black, text=white] (v6) at (1.5, -5.5) {};
  \node[circle, draw, fill=black, text=white] (v7) at (3, -4.5) {};
  \node[circle, draw, fill=black, text=white] (v8) at (1.5, -3.5) {};
  
  \begin{scope}[very thick,decoration={
    markings,
    mark=at position 0.6 with {\arrow[scale = 1.5]{>}}}
    ] 
  \foreach \i/\j in {2/1,3/2,3/8}
    \draw[line width=1.5pt, red, postaction = {decorate}] (v\i) -- (v\j);
    \foreach \i/\j in {3/4,4/5,5/6,6/7,7/8}
    \draw[line width=1.5pt, postaction = {decorate}] (v\i) -- (v\j);

    \end{scope}

\node at (-.65,.35) {\Large{$v_1$}};
\node at (-.65,-3) {\Large{$v_i$}};
\node at (1.75,-3) {\Large{$v_k$}};
    
\end{tikzpicture}
\caption{An out-lasso and a flipped out-lasso.}
\label{fig:lassodef}
\end{figure}

Suppose that we have graph $G$ with orientation $D$ and that within this we have found a lasso $L=(v_1, \ldots v_k; v_i)$ or flipped lasso $L'=(v_i; v_1, \ldots v_k)$. We say that $L$ or $L'$ is of \emph{type} $(a,b,c)$ (with respect to $D$) if $v_1\in D_a$, $v_i\in D_b$, and $v_k \in D_c$. 

Suppose that we have we have found a lasso $L=(v_1, \ldots v_k; v_i)$ of type $(a, b, c)$ in a graph $G$ with orientation $D$. Suppose then that we flip $L$ and obtain the flipped lasso $L'$, with $D'$ the corresponding modified orientation of $G$. Note that the out-degree of every vertex of $G$ is unchanged from $D$ to $D'$, except for the vertices $v_1, v_i, v_k$. If $L$ is an out-lasso, then we know that $L'$ is of type $(a-1, b+2, c-1)$ (with respect to $D'$); if $L$ is an in-lasso, then we know that $L'$ is of type $(a+1, b-2, c+1)$ (with respect to $D'$). 

The idea of flipping a lasso is only helpful if we can guarantee the existence of a lasso in the first place. The following lemma helps with that. Here, and in what follows, given a graph $G$ with orientation $D$, for all $v\in V(G)$ we denote by $S_v$ the set of all vertices reachable from $v$ in $D$ via a directed path. Similarly, we denote by $T_v$ the set of all vertices which can reach $v$ in $D$ via a directed path. A vertex $v$ in an oriented graph is a \emph{source} if all its incident edges point out of $v$; $v$ is a \emph{sink} if all its incident edges point into $v$.

\begin{lemma}\label{lem: lasso} Let $G$ be a graph with orientation $D$ and suppose there exists $v\in V(G)$ with $S_v \cap T_v = \{v\}$. If $S_v$ contains no sinks, then there exists an out-lasso starting at $v$. If $T_v$ contains no sources, then there exists an in-lasso starting at $v$. 
\end{lemma}

\begin{proof}
Both statements of the lemma can be proved in a similar way; we just prove the first. 
Take a longest directed path $P$ starting at $v$, say $P=(v=v_1,\ldots, v_k)$. Since $v\in S_v$ trivially and $S_v$ contains no sinks, we know that $k\geq 2$. Since $v_k\in S_v$ we know that $v_k$ is not a sink. Since $P$ is however a longest path, we know that $v_k$ has an out-neighbour $v_i$ for some $i\in\{1, \ldots, k-1\}$. In fact, $i\geq 2$, since if $i=1$ then $v_2\in S_v\cap T_v$. In particular, this implies that $k\geq 3$. But now the path $P$, plus the edge $v_kv_i$, gives an out-lasso starting at $v$. 
\end{proof}

We can now prove the following.

\setcounter{theorem}{5}
\begin{theorem}
Let $G$ be a graph and let $F: V(G)\to 2^{\mathbb{N}}$ be such that for each $v\in V(G)$: all holes have size at most two; between any pair of distinct holes is a home of size at least 3, and each end-interval must be either a hole of size one or a home of size at least 2. Then $G$ admits an $F$-avoiding orientation.
\end{theorem}
\setcounter{theorem}{12}

\begin{proof}
For each $v\in V(G)$ we partition $\{0,\ldots, d(v)\}$ into four sets, one being $F(v)$ and the other three as follows: \begin{align*}
   A(v) &= \{i: i\notin F(v), i-1\in F(v)\}\cap \{0,\ldots, d(v)\} \\
   B(v) &= \{i: i\notin F(v), i+1\in F(v)\}\cap \{0,\ldots, d(v)\} \\
   X(v) &= \{0,\ldots, d(v)\}\setminus (A(v)\cup B(v) \cup F(v))
\end{align*} 
Note that for any given vertex $v$, the sets $A(v)$ and $B(v)$ are precisely the integers in homes of $v$ which are adjacent to some hole (Above and Below respectively). Moreover, $A(v)\cap B(v) = \emptyset$ since no homes have size 1. Similarly we can think of $X(v)$ as the integers of $\{0,\ldots, d(v)\}$ which are \emph{far} from (meaning not adjacent to) a hole. For an orientation $D$ of $G$, recall that
$D_F = \{v\in V(G) : d^+_D(v)\in F(v)\}$; we now extend this definition so that we may also refer to $D_A, D_B, D_X$, that is, those vertices $v\in V(G)$ with out-degree in the set $A(v), B(v)$, or $X(v)$, respectively.

Let $D$ be an orientation which minimizes $|D_F|$ and subject to that, maximizes $|D_X|$. If $|D_F|=0$ then $D$ is an $F$-avoiding orientation, so suppose for a contradiction that $|D_F|\neq 0$; hence we may choose $v\in D_F$. Note that since all holes of $v$ are of size at most $2$, and we cannot have a hole of size two as en end-interval, either $d^+_D(v)+1 \in A(v)$, or  $d^+_D(v)-1 \in B(v)$. We shall consider both cases.

\textbf{Case 1: $d^+_D(v) + 1 \in A(v)$}. We wish to increase the out-degree of $v$ by one, moving $v$ into a home. To this end, note that there exists $w\in T_v\setminus \{v\}$, since $d^+_D(v)+1\in A(v)$ implying that $d^+_D(v)\neq d(v)$. Choose some directed path from $w$ to $v$ in $D$, and let $D'$ be the orientation obtained from $D$ by reversing the direction of every edge of this path. Note $d^+_{D'}(v) = d^+_D(v) +1$ so $v\notin D'_F$ and thus by minimality of $|D_F|$, $w\in D'_F$ and $w\notin D_F$. Since $d^+_{D'}(w) = d^+_D(w) - 1$, this means that $w\in D_A$. Moreover, $T_v\setminus \{v\}\subseteq D_A$.

Suppose now that there exists some $u\in S_v \cap T_v \setminus \{v\}$. Since $u\in T_v\setminus \{v\}$, we also know that $u\in D_A$. Since $u\in S_v\setminus \{v\}$, we can let $D''$ be the orientation obtained from $D$ by reversing the direction of each edge in some directed path from $v$ to $u$. Then $d^+_{D''}(v) = d^+_D(v)-1$ and $d^+_{D''}(u) = d^+_D(u)+1$. Since $u\in D_A$ this means that $u\in D''_X$ (since all holes between homes have size at least three). We may have $v\in D''_F$ or not, depending on whether the hole is size one or two. But in any case, we have $|D''_F|\leq |D_F|$ and $|D''_X|>|D_X|$, and we obtain a contradiction. Thus, we may assume $S_v \cap T_v = \{v\}$.

Note that $T_v\setminus\{v\}$ contains no sources, since $T_v\setminus\{v\}\subseteq D_A$, and since end-intervals cannot be homes of size one. Thus, by \cref{lem: lasso} there exists an in-lasso starting at $v$ of type $(d^+_D(v), d^+_D(w_1), d^+_D(w_2))$ for some $w_1,w_2\in T_v\setminus \{v\}$ which implies $w_1,w_2\in D_A$. Let $D^*$ be the orientation obtained from $D$ by flipping the in-lasso. Then the flipped in-lasso is of type $(d^+_D(v)+1, d^+_D(w_1)-2, d^+_D(w_2)+1)$ with respect to $D^*$. By the assumption of this case, this means that $v\in D^*_A$. Since $w_2\in D_A$, and between any two distinct holes is a home of size at least $3$, we also get that $w_2\in D^*_X$. We may have $w_1\in D^*_F$ or not, depending on whether the hole is size one or two. However, in any case,  $|D^*_F|\leq |D_F|$ and $|D^*_X| > |D_X|$, and we obtain a contradiction.

\textbf{Case 2:} $d^+_D(v) - 1 \in B_F(v)$. We wish to decrease the out-degree of $v$ by one, moving $v$ into a home. To this end, note that there exists $u\in S_v\setminus \{v\}$, since $d^+_D(v)-1\in B(v)$ implying that $d^+_D(v)\neq 0$. Choose some directed path from $v$ to $u$ in $D$, and let $D'$ be the orientation obtained from $D$ by reversing the direction of every edge of this path. Note $d^+_{D'}(v) = d^+_D(v) -1$ so $v\notin D'_F$ and thus by minimality of $|D_F|$, $u\in D'_F$ and $u\notin D_F$. Since $d^+_{D'}(u) = d^+_D(u) + 1$, this means that $u\in D_B$. Moreover, $S_v\setminus \{v\}\subseteq D_B$.

Suppose now that there exists some $w\in S_v \cap T_v \setminus \{v\}$. Since $w\in S_v\setminus \{v\}$, we also know that $w\in D_B$. Since $w\in S_v\setminus \{v\}$, we can let $D''$ be the orientation obtained from $D$ by reversing the direction of each edge in some directed path from $v$ to $w$. Then $d^+_{D''}(v) = d^+_D(v)-1$ and $d^+_{D''}(w) = d^+_D(w)+1$. Since $w\in D_B$ this means that $w\in D''_X$ (since all holes between homes have size at least three). We may have $v\in D''_F$ or not, depending on whether the hole is size one or two. But in any case, we have $|D''_F|\leq |D_F|$ and $|D''_X|>|D_X|$, and we obtain a contradiction. Thus, we may assume $S_v \cap T_v = \{v\}$.

Note that $S_v\setminus\{v\}$ contains no sinks, since $S_v\setminus\{v\}\subseteq D_B$, and since end-intervals cannot be homes of size one. Thus, by \cref{lem: lasso} there exists an out-lasso starting at $v$ of type $(d^+_D(v), d^+_D(u_1), d^+_D(u_2))$ for some $u_1,u_2\in S_v\setminus\{v\}$. Let $D^*$ be the orientation obtained from $D$ by flipping the out-lasso. Then the flipped out-lasso is of type $(d^+_D(v)-1, d^+_D(u_1)+2, d^+_D(u_2)-1)$ with respect to $D^*$. By the assumption of this case, this means that $v\in D^*_B$. Since $u_2\in D_B$, and between any two distinct holes is a home of size at least $3$, we also get that $u_2\in D^*_X$. We may have $u_1\in D^*_F$ or not, depending on whether the hole is size one or two. However, in any case,  $|D^*_F|\leq |D_F|$ and $|D^*_X| > |D_X|$, and we obtain a contradiction.
\end{proof}

We can now use \cref{thm: 2hole3home} (and Theorem \ref{lem: extreme}) to  validate \cref{conj: regConst} for $5$- and $6$-regular graphs. Our methods work for all cases of the following theorem, although when all holes of are of size one, we will just quote the after-mentioned result of Ma and Lu \cite{no-consecutive}  (namely that Conjecture \ref{conj: F-avoiding} holds when all holes have size 1), for the sake of efficiency.

\setcounter{theorem}{1}
\begin{theorem}
Let $G$ be a $d$-regular graph with $d\geq 5$, and let $F\subseteq\{0, 1, 2, \ldots, d\}$ with $|F|\leq 2$. Then $G$ admits an $F$-avoiding orientation.
\end{theorem}

\begin{proof} Let $G'$ be the graph obtained from $G$ by adding a dummy vertex and joining it to every vertex of odd degree in $G$. There are an even number of these, so $G'$ is Eulerian, and can be decomposed into cycles. By orienting each of these cycles consistently we get an orientation $D'$ of $G'$ where $d^+_{D'}(v)=d_{D'}^-(v)$ for all $v\in V(G')$. Letting $D$ be the restriction of $D'$ to $G$ we get that $d^+_{D}(v)\in\{\lfloor\frac{d}{2}\rfloor, \lceil\frac{d}{2}\rceil\}$ for all $v\in V(G)$. 

As discussed above, following Ma and Lu \cite{no-consecutive} we may assume that $F=\{x, x+1\}$ for some $x\in\{0,1, \ldots, d\}$. So, if $D$ is not an $F$-avoiding orientation then $\{x, x+1\}\cap\{\lfloor\frac{d}{2}\rfloor, \lceil\frac{d}{2}\rceil\}$ is nonempty. 

Suppose first that $x= \lfloor\frac{d}{2}\rfloor$. Then $F$ consists of one hole of size two and two end-homes of sizes $\lfloor\frac{d}{2}\rfloor$ and $\lceil\frac{d}{2}\rceil-1$, respectively. Since $d\geq 5$ these end-homes are both of size at least two, and hence by Theorem \ref{thm: 2hole3home} $G$ admits an $F$-avoiding orientation.

Suppose next that $x=\lfloor\frac{d}{2}\rfloor-1$. Then $F$ consists of one hole of size two and two end-homes of sizes $\lfloor\frac{d}{2}\rfloor-1$ and $\lceil\frac{d}{2}\rceil$, respectively. If $d\geq 6$ then these end-home are both of size at least two, and again Theorem \ref{thm: 2hole3home} implies that $G$ admits an $F$-avoiding orientation. If $d=5$ then $F=\{1,2\}$ and we get that $G$ admits an $F$-avoiding orientation by Theorem \ref{lem: extreme}.

We may now assume that $x=\lceil\frac{d}{2}\rceil=\tfrac{d+1}{2}$ (so $d$ is odd). Then $F$ consists of one hole of size two and two end-homes of sizes $\tfrac{d+1}{2}$ and $\tfrac{d-3}{2}$, respectively. If $d\geq 7$ these end-intervals are both of size at least two, and again Theorem \ref{thm: 2hole3home} implies that $G$ admits an $F$-avoiding orientation. If $d=5$ then $F=\{3,4\}$. Apply Theorem \ref{lem: extreme} to get an orientation $D^*$ that is $\{1,2\}$-avoiding. Then reverse the direction of every edges in $D^*$. This new orientation of $G$ is $\{3, 4\}$-avoiding, completing our proof. 
    \end{proof}

\bibliographystyle{abbrv}
\bibliography{bib}

\end{document}